\documentclass[9pt]{amsart}
\usepackage{amssymb, url, color}
\usepackage[all]{xy}
\usepackage{mathpazo} 
\usepackage{avant}

\usepackage[margin=4cm]{geometry}

\newcommand{\onto}{\twoheadrightarrow}

\newcommand{\Q}{\mathbb{Q}}
\newcommand{\Z}{\mathbb{Z}}

\newcommand{\G}{\Gamma}

\newcommand{\la}{\langle}
\newcommand{\ra}{\rangle}

\newcommand{\SL}{\mathrm{SL}}
\newcommand{\GL}{\mathrm{GL}}

\newcommand{\E}{\mathrm{E}}
\newcommand{\F}{\mathrm{F}}

\newtheorem{thm}{Theorem}

\newtheorem{prop}[thm]{Proposition}

\theoremstyle{definition}

\newtheorem*{notation}{Notation}

\begin{document}
\title[A true relative of Suslin's normality theorem]{\Large A true relative of Suslin's normality theorem}
\author{Bogdan Nica}
\begin{abstract}
We prove a normality theorem for the ``true'' elementary subgroups of $\SL_n(A)$ defined by the ideals of a commutative unital ring $A$. Our result is an analogue of a normality theorem, due to Suslin, for the standard elementary subgroups, and it greatly generalizes a theorem of Mennicke.
\end{abstract}
\address{Mathematisches Institut, Georg-August Universit\"at G\"ottingen}
\email{bogdan.nica@gmail.com}
\date{\today}

\maketitle

\section{Introduction}
Let $A$ be a unital commutative ring, and let $n\geq 2$. The elementary subgroup $\E_n(A)$ is the subgroup of $\SL_n(A)$ generated by the elementary matrices $e_{ij}(a)=1_n+ae_{ij}$, where $i\neq j$ and $a\in A$. Studying the size of $\E_n(A)$ in $\SL_n(A)$ is fraught with surprises and subtleties, but there is one fact which is reassuringly general: if $n\geq 3$ then $\E_n(A)$ is normal in $\SL_n(A)$. This is \emph{Suslin's Normality Theorem} \cite{Sus}.

Ideals define relative elementary subgroups. The \textbf{normal elementary subgroup} $\E_n(\pi)$ corresponding to an ideal $\pi$ of $A$ is the normal subgroup of $\E_n(A)$ generated by the elementary matrices with coefficients in $\pi$:
 \begin{align*}
 \E_n(\pi)=\la\!\la e_{ij}(a): a\in \pi, i\neq j\ra\!\ra_{\E_n(A)}
 \end{align*} 

In fact, Suslin proved the following relative normality theorem. What we have referred to as the Normality Theorem is the absolute case $\pi=A$.

\begin{thm}[Suslin]\label{Suslin}
Let $n\geq 3$. Then $\E_n(\pi)$ is normal in $\SL_n(A)$.
\end{thm}

Taking just the subgroup closure of the elementary matrices with coefficients in an ideal $\pi$ gives rise to the \textbf{true elementary subgroup} $\F_n(\pi)$: 
\begin{align*}
\F_n(\pi)=\la e_{ij}(a): a\in \pi, i\neq j\ra
\end{align*}
The notation is the one proposed by Tits \cite{Tits}, while the terminology is ours. Our main result is a relative normality theorem for the true elementary subgroups. Once again, the Normality Theorem appears as the absolute case $\pi=A$.

\begin{thm}\label{N}
Let $n\geq 3$. Then $\F_n(\pi)$ is normal in the subgroup $\{g \in \SL_n(A): g \textup{ mod } \pi \textup{ is diagonal}\}$.
\end{thm}

Warming up the cold truth of the theorem is a short story explaining how we were led to this result. The setting is the most familiar case, $A=\Z$. Then the absolute elementary subgroup $\E_n(\Z)$ exhausts $\SL_n(\Z)$. But what becomes of the relative elementary groups?

In 1965, Mennicke \cite{Men1} proved the following remarkable fact: for $n\geq 3$, the normal elementary subgroup $\E_n(N)=\la\!\la e_{ij}(N): i\neq j\ra\!\ra$ coincides with the principal congruence subgroup $\G_n(N)=\{g \in \SL_n(\Z): g \equiv 1_n\; \mathrm{mod}\; N\}$. The proof is elementary, though somewhat intricate. Using earlier observations of Brenner, Mennicke was then able to derive the \emph{Normal Subgroup Theorem} stating that every normal subgroup of $\SL_n(\Z)$, $n\geq 3$, is either central or it contains a principal congruence subgroup. The \emph{Congruence Subgroup Property}, that every finite-index subgroup of $\SL_n(\Z)$, $n\geq 3$, contains a principal congruence subgroup, is an immediate consequence.  

More than three decades later, in 2000, Mennicke \cite{Men2} published the following counterpart of his 1965 result: for $n\geq 3$, the true elementary subgroup $\F_n(N)=\la e_{ij}(N): i\neq j\ra$ coincides with the congruence subgroup $\Delta_n(N)=\{g \in \SL_n(\Z): g \equiv 1_n\; \mathrm{mod}\; N, \; g_{ii} \equiv 1 \; \mathrm{mod}\; N^2\}$. Mennicke's approach to this more recent - in fact, surprisingly recent - theorem is significantly more complicated. It is, in a sense, an unnecessary proof: a short and conceptual argument shows that Mennicke's two theorems on elementary subgroups - Theorem E and Theorem F, so to speak - are equivalent. This equivalence can be formulated in the generality of commutative rings, the key ingredient being a 1976 theorem of Tits, stated below. As I have recently learnt from Andrei Rapinchuk, the equivalence was noted soon after \cite{Men2} was circulated in preprint form. To the best of my knowledge, the equivalence does not appear in the literature, so it seems worthwhile to discuss it. As it turns out, Theorem~\ref{N} will fit in this discussion.

Here is Tits' theorem, the $\SL_n$ case of a result proved in \cite{Tits} for Chevalley groups. Incidentally, we expect that our Theorem~\ref{N} can be generalized, as well, in the context of Chevalley groups.

\begin{thm}[Tits]\label{Tits}
Let $n\geq 3$. Then $\E_n(\pi^2)$ is contained in $\F_n(\pi)$.
\end{thm}

Mennicke's complicated proof of the identity $\F_n(N)=\Delta_n(N)$ can be circumvented, but there is something worth saving. Most of Mennicke's argument in \cite{Men2} goes into proving the weaker statement that $\F_n(N)$ is normalized by $\Delta_n(N)$ \cite[Thm.2]{Men2}. In a subsequent remark, Mennicke adds that such a normality relation holds, more generally, for arithmetic Dedekind rings. Our main theorem proves much more: normality holds in any commutative ring, and with respect to a larger congruence subgroup. The proof is also much simpler.

\section{The Suslin factorization and applications}

Following Suslin \cite{Sus}, we factorize conjugates of elementary matrices into products of elementary matrices and ``suspended'' $\SL_2$ matrices. 

Let $g \in \SL_n(A)$, where $n\geq 2$, let $a\in A$, and fix $i\neq j$. We start by writing
\begin{align*}
g^{-1}e_{ij}(a)g=1_n+a(g^{-1}e_{ij}g)=1_n+a\cdot vw
\end{align*}
where $v$ is the $i$-th column of $g^{-1}$, and $w$ is the $j$-th row of $g$. If we further let $w'$ denote the $i$-th row of $g$, then it can be verified that
\begin{align*}
w=\sum_{k<l} c_{kl} (v_l e_k-v_k e_l), \qquad c_{kl}:=w_kw'_l-w_lw'_k=g_{jk}g_{il}-g_{jl}g_{ik}
\end{align*}
where $e_k=(0,\dots,1\dots,0)$ is the $k$-th basic row vector.
As $(v_l e_k-v_k e_l)v=0$, we get
\begin{align*}
g^{-1}e_{ij}(a)g=1_n+\sum_{k<l} ac_{kl}\cdot v(v_l e_k-v_k e_l)= \prod_{k<l} \Big(1_n+ac_{kl} \cdot v(v_l e_k-v_k e_l)\Big).
\end{align*}
Next, we decompose each factor $1_n+ac_{kl} \cdot v(v_l e_k-v_k e_l)$ as
\begin{align*}
\Big(  1_n & +ac_{kl} \cdot (v_k e^k+v_l e^l) (v_l e_k-v_k e_l)\Big)\prod_{s\neq k,l}\Big( 1_n+ac_{kl} v_s e^s(v_l e_k-v_k e_l)\Big)\\
&=\begin{pmatrix}1+ ac_{kl} v_kv_l & -ac_{kl} v_k^2\\ ac_{kl} v_l^2 & 1- ac_{kl} v_kv_l
\end{pmatrix}^{* kl} \prod_{s\neq k,l}\big( 1_n+ac_{kl} v_s v_l \cdot e_{sk}\big)\big( 1_n-ac_{kl} v_s v_k \cdot e_{sl}\big).
\end{align*}
In the first displayed line, $e^k$ denotes the $k$-th basic column vector. In the second displayed line, $\left( \begin{smallmatrix} x&y\\ z&t \end{smallmatrix} \right)^{* kl}\in \SL_n(A)$ is the \textbf{$(k,l)$-suspension} of the matrix $\left( \begin{smallmatrix} x&y\\ z&t \end{smallmatrix} \right)\in \SL_2(A)$, namely the matrix obtained from the identity matrix $1_n$ by grafting $x$ in the $(k,k)$-entry, $y$ in the $(k,l)$-entry, $z$ in the $(l,k)$-entry, respectively $t$ in the $(l,l)$-entry.

\begin{notation}
For $x,y,z\in A$, we define the symbol
\begin{align*}S(x,y;z)=\begin{pmatrix}1+ xyz & -x^2z\\ y^2z & 1- xyz
\end{pmatrix}\in \SL_2(A). \end{align*}
\end{notation}

Then the above computation can be summarized as the following \emph{Suslin factorization}:
\begin{align*}
g^{-1}e_{ij}(a)g=\prod_{k<l} \Big( S(v_k,v_l; ac_{kl})^{* kl} \prod_{s\neq k,l} e_{sk}(ac_{kl} v_s v_l)\: e_{sl}(-ac_{kl} v_s v_k)\Big).
\end{align*}

For the remainder of the section, we assume that $n\geq 3$. Recall that $\pi$ denotes an ideal of $A$. Our commutator convention is that $[g,h]=g^{-1}h^{-1}gh$.

\subsection{Proof of Suslin's Theorem~\ref{Suslin}} Row- and column-reductions lead to the commutator identity
\begin{align*}
\begin{pmatrix}1+ xyz & -x^2z & 0\\ y^2z & 1- xyz & 0\\ 0 & 0 & 1
\end{pmatrix}=\begin{bmatrix}\begin{pmatrix}1 & 0& xz\\ 0 & 1 & yz\\ 0 & 0 & 1
\end{pmatrix},  \begin{pmatrix}1 & 0& 0\\ 0 & 1 & 0\\ y & -x & 1
\end{pmatrix}\end{bmatrix}
\end{align*}
which shows that $S(x,y;z)^*\in\E_n(\pi)$ whenever $z\in \pi$. Here, the notation $S(x,y;z)^*$ stands for any suspension of $S(x,y;z)$. The Suslin factorization yields that $g^{-1}e_{ij}(a)g\in \E_n(\pi)$ whenever $a\in \pi$. 

\subsection{Proof of Tits' Theorem~\ref{Tits}} In fact, we will prove the stronger assertion that $\E_n(\pi^2)$ is contained in the commutator subgroup $[\F_n(\pi), \F_n(\pi)]$. We start from the following commutator identity, generalizing the one of the previous paragraph:
\begin{align*}
\begin{pmatrix}1+ xyz_1z_2 & -x^2z_1z_2 & 0\\ y^2z_1z_2 & 1- xyz_1z_2 & 0\\ 0 & 0 & 1
\end{pmatrix}=\begin{bmatrix}\begin{pmatrix}1 & 0& xz_1\\ 0 & 1 & yz_1\\ 0 & 0 & 1
\end{pmatrix},  \begin{pmatrix}1 & 0& 0\\ 0 & 1 & 0\\ yz_2 & -xz_2 & 1
\end{pmatrix}\end{bmatrix}
\end{align*}
We see that $S(x,y;z_1z_2)^*\in[\F_n(\pi), \F_n(\pi)]$ for $z_1,z_2\in \pi$. Observe, on the other hand, that symbols enjoy the additivity rule $S(x,y;z+z')=S(x,y;z)\: S(x,y;z')$. Therefore $S(x,y;z)^*\in [\F_n(\pi), \F_n(\pi)]$ for $z\in \pi^2$. We also have that $e_{ij}(z)\in [\F_n(\pi), \F_n(\pi)]$ for $z\in \pi^2$ (and $i\neq j$, as usual). This follows from the relations $e_{ij}(u+v)=e_{ij}(u) e_{ij}(v)$, respectively $e_{ij}(uv)=[e_{ik}(u), e_{kj}(v)\big]$ for distinct  $i,j,k$. And so, by the Suslin factorization, we find that $g^{-1} e_{ij}(a)g\in [\F_n(\pi), \F_n(\pi)]$ whenever $a\in \pi^2$.

\subsection{Proof of Theorem~\ref{N}} We claim that $S(x,y;z)^*\in\F_n(\pi)$ whenever $y,z\in \pi$ (or, symmetrically, $x,z\in \pi$). Indeed, row and column operations over $\pi$ (indicated by $r$, respectively $c$) allow for the following transition:

\begin{align*}
&\begin{pmatrix}
1+xyz & -x^2z & 0\\
y^2z & 1-xyz & 0\\
0 & 0 & 1
\end{pmatrix}
\stackrel{c}{\rightsquigarrow}
\begin{pmatrix}
1+xyz & -x^2z & 0\\
y^2z & 1-xyz & 0\\
-y & 0 & 1
\end{pmatrix}
\stackrel{r,r}{\rightsquigarrow}
\begin{pmatrix}
1 & -x^2z & xz\\
0 & 1-xyz & yz\\
-y & 0 & 1
\end{pmatrix}
\stackrel{r}{\rightsquigarrow}\\
&\qquad\qquad\qquad\begin{pmatrix}
1 & -x^2z & xz\\
0 & 1-xyz & yz\\
0 & -x^2yz & 1+xyz
\end{pmatrix}
\stackrel{c,c}{\rightsquigarrow}
\begin{pmatrix}
1 & 0 & 0\\
0 & 1-xyz & yz\\
0 & -x^2yz & 1+xyz
\end{pmatrix}\end{align*}
\smallskip

The last matrix is a suspension of $S(1,x; -yz)$, hence in $\F_n(\pi)$ by what we have learned in the proof of Tits' theorem. We conclude from the Suslin factorization that $g^{-1}e_{ij}(a)g\in \F_n(\pi)$ whenever $a\in \pi$ and $g \in \SL_n(A)$ has all off-diagonal entries in $\pi$.

\section{Congruence subgroups versus elementary subgroups}
The elementary subgroups $\E_n(\pi)$ and $\F_n(\pi)$, corresponding to an ideal $\pi$ of $A$, are relative versions of the absolute elementary subgroup $\E_n(A)$. We will now introduce two subgroups which play a similar role with respect to $\SL_n(A)$. In the particular case of $A=\Z$, we have already encountered them in the Introduction.

The \textbf{principal congruence subgroup} is the subgroup 
\begin{align*}
\G_n(\pi)=\{g \in \SL_n(A): g \equiv 1_n \; \textrm{mod}\; \pi\},
\end{align*} 
that is the kernel of the reduction homomorphism $\SL_n(A)\to \SL_n(A/\pi)$. Its elementary counterpart is the normal elementary subgroup $\E_n(\pi)$. The relation between $\E_n(\pi)$ and $\G_n(\pi)$ has been investigated since the 1960s, and it is crucial for understanding the subgroup structure of $\SL_n(A)$, as well as for the purposes of lower algebraic $K$-theory.

The congruence analogue of the true elementary subgroup $\F_n(\pi)$ is the subgroup 
\begin{align*}
\Delta_n(\pi)=\{g \in \SL_n(A): g \equiv 1_n \;\textrm{mod}\; \pi,\; g_{ii} \equiv 1 \; \textrm{mod}\; \pi^2\}.
\end{align*} It seems quite suggestive to think of $\Delta_n(\pi)$ as the \textbf{secondary congruence subgroup}. 

Suslin's Theorem~\ref{Suslin} says that a normality feature, which is obviously enjoyed by $\G_n(\pi)$ for $n\geq 2$, turns out to be satisfied by $\E_n(\pi)$ as soon as $n\geq 3$. This intuition, that  visible congruence facts have hidden elementary analogues in ``higher rank'', is what guided us towards Theorem~\ref{N}. Indeed, it is easy to check that $\Delta_n(\pi)$ is normal in 
\begin{align*}
\Omega_n(\pi)=\{g \in \SL_n(A): g \textup{ mod } \pi \textup{ is diagonal}\}
\end{align*} 
for $n\geq 2$, by using the diagonal multiplicativity $(gh)_{ii}=g_{ii}\:h_{ii}$ mod $\pi^2$ for $g,h\in \Omega_n(\pi)$. In more detail, Suslin's theorem hinges on having $S(x,y;z)^*\in\E_n(\pi)$ for $z\in \pi$, a higher rank refinement of the obvious fact that $S(x,y;z)\in\G_2(\pi)$ for $z\in \pi$. In the proof of Theorem~\ref{N}, the claim that $S(x,y;z)^*\in\F_n(\pi)$ for $y,z\in \pi$ is suggested by  the obvious fact that $S(x,y;z)\in\Delta_2(\pi)$ for $y,z\in \pi$. 

We may gather the elementary and the congruence subgroups we have defined in the following diagram. The arrows denote inclusions, all being obvious except for the dashed one which is the content of Tits' theorem, and which in addition requires the ``higher rank'' assumption $n\geq 3$.

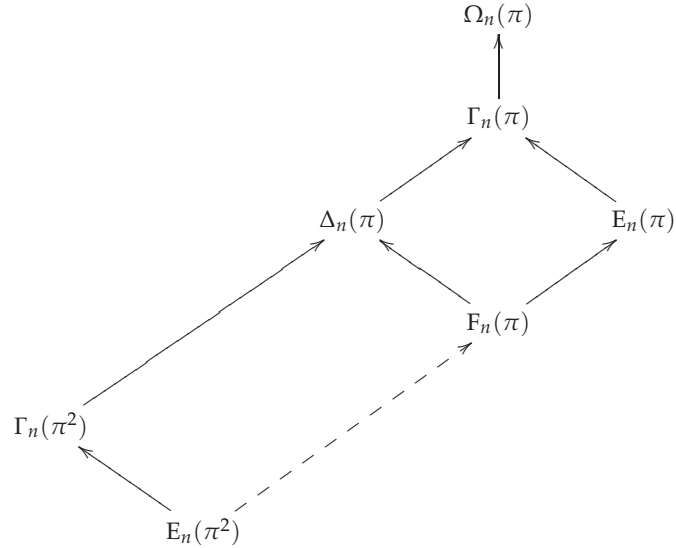
\begin{align*}
\xymatrix{
&&&\Omega_n(\pi) &\\
&&&\G_n(\pi)\ar@{>}[u] &\\
&&\Delta_n(\pi)  \ar@{>}[ru] && \E_n(\pi)\ar@{>}[lu]\\
&&&\F_n(\pi)  \ar@{>}[lu]\ar@{>}[ru]&\\
\G_n(\pi^2)  \ar@{>}[rruu]&&&&\\
&\E_n(\pi^2)  \ar@{>}[lu]\ar@{-->}[rruu]&&&
}
\end{align*}
\smallskip

Consider the reduction homomorphism $r: \G_n(\pi)\to \mathfrak{gl}_n(\pi/\pi^2)$, given by $g\mapsto g-1_n$ mod $\pi^2$. Here $\mathfrak{gl}_n(\pi/\pi^2)$ denotes the additive group of $n\times n$ matrices over $\pi/\pi^2$. As 
\begin{align*}
1 \equiv \prod g_{ii}\equiv 1+\sum (g_{ii}-1) \textrm{ mod }\pi^2
\end{align*} for each $g\in \G_n(\pi)$, the range of $r$ lies within the zero-trace subgroup $\mathfrak{sl}_n(\pi/\pi^2)$. Now $\mathfrak{sl}_n(\pi/\pi^2)$ is generated by $\{\bar{a} e_{ij}: i\neq j, a\in \pi\}$ together with $\{\bar{a}e_{i+1\: i+1}-\bar{a}e_{ii}: i\neq n, a\in \pi\}$. The off-diagonal generators are visibly in the range of $r$, as $1_n+ae_{ij}\mapsto \bar{a} e_{ij}$. Given $a\in \pi$, note that
\begin{align*}
\begin{pmatrix} 1& 1\\ 0 &1 \end{pmatrix}
\begin{pmatrix} 1& 0\\ a &1 \end{pmatrix}
\begin{pmatrix} 1& -1\\ 0 &1 \end{pmatrix}=
\begin{pmatrix} 1+a& -a\\ a &1-a \end{pmatrix}\in \G_2(\pi)\quad \xrightarrow{\: r\:}\quad\begin{pmatrix} \bar{a} & -\bar{a}\\ \bar{a}& -\bar{a}\end{pmatrix}.
\end{align*}
Taking the $(i, i+1)$-suspension, and keeping in mind that $\bar{a}e_{i\: i+1}$ and $\bar{a}e_{i+1\: i}$ are already in the image of $r$, we see that the diagonal differences are in the range of $r$. To conclude, the reduction homomorphism $r: \G_n(\pi)\to \mathfrak{sl}_n(\pi/\pi^2)$ is onto, with kernel $\G_n(\pi^2)$. (As an aside, let us point out that $\Delta_n(\pi)$ is the preimage of the zero-diagonal subgroup of $\mathfrak{sl}_n(\pi/\pi^2)$, so $\G_n(\pi)/\Delta_n(\pi)\simeq (\pi/\pi^2,+)^{n-1}$ and $\Delta_n(\pi)/\G_n(\pi^2)\simeq (\pi/\pi^2,+)^{n^2-n}$. This relative position of $\Delta_n(\pi)$ explains the longer arrows in our diagram. We leave it to the interested reader to check that $\Omega_n(\pi)/\G_n(\pi)\simeq (\GL_1(A/\pi))^{n-1}$.)

As the previous argument shows, $r(\E_n(\pi))=r(\G_n(\pi))$ and $r(\F_n(\pi))=r(\Delta_n(\pi))$. Thus the elementary subgroups are first-order approximants of the congruence subgroups, in the sense that
\begin{align*}
\G_n(\pi)=\E_n(\pi)\cdot \G_n(\pi^2), \qquad \Delta_n(\pi)=\F_n(\pi)\cdot \G_n(\pi^2).
\end{align*}
We infer that the natural homomorphism $\G_n(\pi^2)/\E_n(\pi^2)\to\G_n(\pi)/\E_n(\pi)$ is onto for $n\geq 2$. When $n\geq 3$, Tits' theorem allows us to insert the coset space $\Delta_n(\pi)/\F_n(\pi)$ in between, so that we have the following:
\begin{thm}\label{squeeze} Let $n\geq 3$. Then the inclusions $\G_n(\pi^2)\subseteq \Delta_n(\pi)\subseteq\G_n(\pi)$ induce surjections
\begin{align*}
\G_n(\pi^2)/\E_n(\pi^2)\onto\Delta_n(\pi)/\F_n(\pi)\onto\G_n(\pi)/\E_n(\pi).
\end{align*}
\end{thm}

In particular, for $n\geq 3$, the property that $\E_n(\pi)=\G_n(\pi)$ for every ideal $\pi$ is equivalent to the property that $\F_n(\pi)=\Delta_n(\pi)$ for every ideal $\pi$. This is the conceptual explanation, promised in the Introduction, for the equivalence between Mennicke's two theorems.
 
So far, the discussion did not involve Theorem~\ref{N}. But our theorem does have something to add to Theorem~\ref{squeeze}, namely the fact that $\Delta_n(\pi)/\F_n(\pi)$ is actually a quotient \emph{group}, and not just a coset space.

Theorem~\ref{squeeze} can be applied, for instance, when $A$ is the ring of integers in a number field $K$. The Bass - Milnor - Serre solution \cite{BMS67} to the Congruence Subgroup Problem for $\SL_n$, $n\geq 3$, establishes that each quotient $\G_n(\pi)/\E_n(\pi)$ is a finite cyclic group whose order divides the number of roots of unity in $K$, and which is furthermore trivial if $K$ admits a real embedding. The same is then true for the quotient $\Delta_n(\pi)/\F_n(\pi)$.

\section{The case of $\SL_2$} 
For $n=2$, both Theorem~\ref{Tits} and Theorem~\ref{N} fail in general. However, there is some occasional truth to them. 

We illustrate the first point in the familiar case $A=\Z$. 

\begin{prop}
In $\SL_2(\Z)$, the following hold for $N\geq 4$: 
\begin{itemize}
\item $\F_2(N)$ is not normal in $\Delta_2(N)$,
\item $\E_2(N^2)$ is not contained in $\F_2(N)$.
\end{itemize}
\end{prop}

\begin{proof}
Let
\begin{align*}
\alpha=\begin{pmatrix} 1&1\\ 0&1 \end{pmatrix}, \qquad \beta=\begin{pmatrix} 1&0\\ N&1 \end{pmatrix}
\end{align*} and note that $\alpha$ and $\beta$ generate a free group of rank $2$, since they can be simultaneously conjugated into $\left( \begin{smallmatrix} 1&\sqrt{N}\\ 0&1 \end{smallmatrix} \right)$ and $\left( \begin{smallmatrix} 1&0\\ \sqrt{N}&1 \end{smallmatrix} \right)$. Put
\begin{align*}
\omega= \begin{pmatrix} 1& 1\\ 0 &1 \end{pmatrix}
\begin{pmatrix} 1& 0\\ N^2 &1 \end{pmatrix}
\begin{pmatrix} 1& -1\\ 0 &1 \end{pmatrix}=\begin{pmatrix} 1+N^2& -N^2\\ N^2 &1-N^2 \end{pmatrix}.
\end{align*}
Then $\omega\in \E_2(N^2)$ but $\omega\notin \F_2(N)$, as $\alpha\beta^N\alpha^{-1}$ is not a word in $\alpha^N$ and $\beta$. Also $\omega\in \Delta_2(N)$ while $\omega^{-1}\alpha^N\omega\notin \F_2(N)$ since $\alpha \beta^{-N}\alpha^N\beta^N\alpha^{-1}$ is not a word in $\alpha^N$ and $\beta$.
\end{proof}

On the other hand, we have the following:

\begin{thm}[Vaserstein]\label{Vas}
Let $A$ be the ring of integers in a number field $K$ which is neither the rational field $\Q$, nor an imaginary quadratic field $\Q(\sqrt{-D})$, and let $\pi$ be an ideal in $A$. Then:
\begin{itemize}
\item $\F_2(\pi)$ is normal in $\Delta_2(\pi)$,
\item $\E_2(\pi^2)$ is contained in $\F_2(\pi)$, 
\item the natural homomorphism $\G_2(\pi^2)/\E_2(\pi^2)\to\Delta_2(\pi)/\F_2(\pi)$ is an isomorphism.
\end{itemize}
\end{thm}
 
This is taken from \cite{Vas72}. A gap in Vaserstein's paper was later corrected by Liehl \cite{Lie}.

Vaserstein's theorem conforms with the principle that, over an arithmetic ring with infinitely many units, $\SL_2$ should behave like $\SL_n$ with $n\geq 3$. This principle was born with Serre's solution \cite{Ser70} to the Congruence Subgroup Problem for $\SL_2$. In the case of a ring of integers $A$ as in Theorem~\ref{Vas}, Serre's results say that each quotient $\G_2(\pi)/\E_2(\pi)$ is a finite cyclic group whose order divides the number of roots of unity in $K$, and which is furthermore trivial if $K$ admits a real embedding. By Vaserstein's theorem, the same statement applies to the quotient $\Delta_2(\pi)/\F_2(\pi)$.

\bigskip
\noindent\textbf{Acknowledgements.} Constructive suggestions from Andrei Rapinchuk, feedback and the occasional espresso from Laurent Bartholdi, respectively the support of the Alexander von Humboldt Foundation are all gratefully acknowledged.

\end{document}